\documentclass{amsart} 

\usepackage{etex}
\usepackage{lmodern}
\usepackage[T1]{fontenc}
\usepackage[utf8]{inputenc}

\usepackage{amsthm, amsmath, amssymb, amsfonts}
\usepackage{hyperref, cleveref}

\usepackage{mathrsfs}

\usepackage{mathtools} 

\usepackage{graphicx}
\usepackage{float, wrapfig}
\usepackage[all]{xy} 

\usepackage{xcolor}

\usepackage{listings}
\usepackage{enumerate}
\usepackage{multirow}

\theoremstyle{plain}
\newtheorem{Theorem}{Theorem}[section]
\newtheorem*{Theorem*}{Theorem}
\newtheorem*{Proposition*}{Proposition}
\newtheorem{Proposition}[Theorem]{Proposition}
\newtheorem{Lemma}[Theorem]{Lemma}

\newtheorem{Corollary}[Theorem]{Corollary}

\newtheorem*{Problem*}{Problem}

\theoremstyle{definition}
\newtheorem*{Example*}{Example}
\newtheorem*{Definition*}{Definition}

\theoremstyle{remark}
\newtheorem*{Remark}{Remark}

\newcommand{\CC}{\mathbb{C}}
\newcommand{\QQ}{\mathbb{Q}}
\newcommand{\FF}{\mathbb{F}}
\newcommand{\ZZ}{\mathbb{Z}}

\newcommand{\NN}{\mathbb{N}}
\newcommand{\KK}{\mathbb{K}}

\newcommand{\Fp}{\mathbb{F}_p}

\newcommand{\Mod}{\mathrm{mod} \ }
\DeclareMathOperator{\res}{Res}

\numberwithin{equation}{section} 
\newcommand\numberthis{\addtocounter{equation}{1}\tag{\theequation}} 

\title{Graded character rings of finite groups}
\author{B\'eatrice I. Chetard}
\date{\today}

\begin{document}

	\begin{abstract}
	Let $G$ be a finite group. The ring $R_\KK(G)$ of virtual characters of $G$ over the field $\KK$ is a $\lambda$-ring; as such, it is equipped with the so-called $\Gamma$-filtration, first defined by Grothendieck. We explore the properties of the associated graded ring $R^*_\KK(G)$, and present a set of tools to compute it through detailed examples. In particular, we use the functoriality of $R^*_\KK(-)$, and the topological properties of the $\Gamma$-filtration, to explicitly determine the graded character ring over the complex numbers of every group of order at most $8$, as well as that of dihedral groups of order $2p$ for $p$ prime.
	\end{abstract}

	\maketitle

	\section{Introduction}
Let $G$ be a finite group and $\KK$ a field of characteristic zero. The character ring $R_\KK(G)$ of $G$ is the abelian group generated by the irreducible characters of $G$ over $\KK$, which has a ring structure coming from the tensor product of representations. The ring $R_\KK(G)$, with the operations $\{\lambda^n: R_\KK(G) \to R_\KK(G)\}$ induced by exterior powers of representations, together satisfy the axioms of a $\lambda$-ring. \par
Grothendieck used the theory of $\lambda$-rings in the 1960s, to provide a categorical framework for the Riemann-Roch theorem (see \cite{berthelot-sga6}). With each $\lambda$-ring $R$, he associated a filtration (hereafter the Grothendieck filtration, or $\Gamma$-filtration); the associated graded ring $gr^* R$ is equipped with so-called algebraic Chern classes, which satisfy the properties of the eponymous construction in algebraic topology. To underline the importance of this construction, let us mention that, when $X$ is a smooth algebraic variety (say, over the complex numbers), there is an isomorphism
\[
	gr^*K(X) \otimes \QQ \cong CH^*X \otimes \QQ,
\]
where $CH^*X$ is the Chow ring of $X$ and $K(X)$ is the Grothendieck group of algebraic vector bundles over $X$ (see for example \cite[Ex. 15.2.16]{fulton-intersection}). If $X$ is a reasonable topological space and $K(X)$ is, this time, its topological $K$-theory, then
\[
	gr^*K(X) \otimes \QQ \cong H^{2*}(X,\QQ),
\]
where $H^{2*}(X,\QQ)$ is the even part of the singular cohomology of $X$ (see \cite[Prop. 3.2.7]{atiyah-k-theory}). Both of these isomorphisms are compatible with Chern classes.\par
Character rings are natural examples of $\lambda$-rings; in the sequel we write $R^*_\KK(G)$ for $gr^*(R_\KK(G))$. There only are a few examples of explicit computations of $R_\KK^*(G)$ in the literature: the first paper on the subject is a 2001 preprint by Beauville (\cite{beauville}), who showed that for a complex connected reductive group $G$, the graded ring $R_\CC^*(G)\otimes \QQ$ is simply described in terms of a maximal torus and its Weyl group.\par
It should be noted that in all of the above results, the graded ring is tensored with the rational field. However, we show in \Cref{R*G_G-torsion} that when $G$ is a finite group, the ring $R^*_\KK(G)$ is torsion, so that $R^*_\KK(G) \otimes \QQ$ is zero in positive degree. Computing $R^*_\KK(G)$ is harder but not hopeless, as is shown by Guillot and Min\'a\v{c} in \cite{guillot-minac}, where they determine explicitly $R^*_\CC(G)\otimes \FF_2$ for some $2$-groups. The present paper extends their work, by providing more explicit examples of graded character rings on the one hand, and by presenting general results about their structure and behaviour on the other hand.\par
The theory of graded character ring is intricate, and studying even the most elementary groups can lead to remarkable results. For instance, the following computation (presented below as \Cref{gradedrepring_Cpk}), which extends \cite[Prop. 3.11]{guillot-minac}, shows that there is no Künneth formula for $R^*_\CC(G)$:
\begin{Theorem}
Let $C_p$ be a cyclic group of prime order $p$. Then:
\[
	R^*_\CC(C_p^k) = \frac{\ZZ[x_1,\cdots,x_k]}{(px_i, x_i^p x_j - x_ix_j^p)}.
\]
\end{Theorem}
This makes determining graded character rings of some "easy" groups surprisingly difficult, a consequence that is in turn illustrated by the rather sophisticated computation of $R^*_\CC(C_4\times C_4)$, the very last one that we present in this paper. It is part of our endeavour to compute the graded character rings of abelian $p$-groups, which are of special interest: we show in \Cref{GH_coprime} that if $G, H$ are of coprime orders, the Künneth formula does hold for $R^*_\CC(G\times H)$. This reduces the computation of graded character rings of abelian groups, to those of abelian $p$-groups. On that topic, we adapt the main theorem of \cite{quillen} to show that for an abelian group $G$ and for each prime $p$, there is an explicit, surjective morphism:
\[
	R_\CC^*(G)\otimes \Fp \to gr_\bullet \Fp G,
\]
where $gr_\bullet \Fp G$ is the graded ring associated to the filtration by powers of the augmentation ideal of $\Fp G$. In particular, the following result is a direct corollary of \Cref{theorem_quillenrelations}:
\begin{Theorem}
	Let $G$ be an abelian $p$-group of the form $C_{p^{i_1}}\times \cdots \times C_{p^{i_n}}$. Then $R^*_\CC(G)$ is generated by elements $x_1, \ldots, x_n$ of degree 1 such that any monomial in any relation between these in $R^*_\CC(G)\otimes
\Fp$ features some $x_k^{p^{i_k}}$, for some index $k$.
\end{Theorem}
Again, this is illustrated by the example of $C_4 \times C_4$:
\begin{Theorem} The graded ring of $C_4 \times C_4$ is given by
\[
	R^*_\CC(C_4\times C_4) = \frac{\ZZ[x,y]}{(4x,4y, 2x^2y + 2xy^2, x^4y^2 - x^2y^4)},
\]
therefore
\[
	R^*_\CC(C_4 \times C_4)\otimes \FF_2 = \frac{\FF_2[x,y]}{(x^4y^2 + x^2y^4)}.
\]
\end{Theorem}
(This is \Cref{Z4Z4} in the text.) Notice how relations modulo $2$ involve either $x^4$ or $y^4$ in each monomial. \par
Most of our results pertain to representations over the complex numbers. The situation can become much more complicated over other fields, as the following result shows:
\begin{Theorem}
	When $G$ is a $p$-group, the graded ring $R^*_\QQ(G)$ is concentrated in degrees multiple of $(p-1)$. In particular, for any prime $p$:
	\[
		R^*_\QQ(C_p) = \frac{\ZZ[x]}{(px)}
	\]
	where $|x| = p-1$.
\end{Theorem}
This statement combines \Cref{prop_gradedringQQ} and \Cref{R*GCp_QQ}. It is striking that, in general, the graded character ring of a cyclic group over the rationals is not known. By contrast, the graded ring $R^*_\CC(G)$ for any cyclic group $G$ was computed in \cite{guillot-minac}, and is generated by an explicit class of degree $1$. \par
The above results provide a convincing argument in favour of graded character rings: their theory, of which we barely scratched the surface, is rich and complex.\par

Another, perhaps more concrete argument in their favour, is that they are a fine invariant of groups: they take into account the $\lambda$-ring structure of $R_\KK(G)$, which can distinguish non-isomorphic groups with the same character table such as the quaternion group $Q_8$ and the dihedral group $D_4$ of order $8$. In fact, \Cref{gradedring_D4} and \Cref{Q8_gradedring} add up to the following:
\begin{Theorem}
\[
	R^*_\CC(D_4) = \frac{\ZZ\left[x,y,b\right]}{(2x,2y,4b,xy,xb-yb)}
\]
with $|x| = |y| = 1$ and $|u| = 2$, and
\[
	R^*_\CC(Q_8) =\frac{\ZZ[x,y,u]}{(2x,2y,8u,x^2,y^2, xy-4u)}
\]
where $\vert x \vert = \vert y \vert = 1$ and $\vert u \vert = 2$.
\end{Theorem}

This paper is organized as follows: the main definitions are introduced in \Cref{definitions}; each of sections $3$ to $6$ is focused on a different computational tool. We show in \Cref{cyclicgroups} two elementary computations, concerning cyclic groups over any algebraically closed $\KK$ (after \cite{guillot-minac}), and over the rationals, and we will then restrict ourselves to the case $\KK = \CC$. In \Cref{restrictions}, we put the cyclic group example to good use: we show that restriction of characters is a well-defined homomorphism and apply it to elementary abelian groups as well as some dihedral groups. In \Cref{quillen}, using a result of Quillen in \cite{quillen}, we construct the aforementioned morphism $R_\CC^*(G)\otimes \Fp \to gr_\bullet \Fp G$. In \Cref{continuity}, we look at the continuity of evaluation of characters with respect to the topology induced by the Grothendieck filtration, and the $p$-adic topology. We apply our results to graded character rings of $p$-groups: first in \Cref{continuity} for the quaternion group of order $8$, and second in \Cref{abelian2groups} to some abelian $2$-groups.\par
In light of the functoriality of $R^*(-)$ and the existence of a restriction morphism, a natural question arising is whether induction of representations also induces a well-defined map between graded rings; in other words, whether $R^*(-)$ is a Mackey functor. This question will be the subject of an second paper, in which we show through an explicit example that $R^*(-)$ is not a Mackey functor, and explore a modified filtration and its associated graded ring.\par
Besides the aforementioned work of Guillot and Min\'{a}\v{c}, which lays much of the groundwork for this paper, here are our main references. All nontrivial statements of representation theory that we use can be found in \cite{serre}; the theory of $\lambda$-rings is outlined in a concise manner in \cite{atiyah-tall}, and the mechanics of the Grothendieck filtration and Chern characters are summed up in \cite{fulton-lang}.\par

\textbf{Acknowledgements.} I would like to thank both my Ph.D. advisors: Pierre Guillot for suggesting this problem in the first place, and for his kind guidance from the first to the last keystroke of this paper, and J\'{a}n Min\'{a}\v{c} for many an enthusiastic and encouraging discussion.

	\section{Definitions and first properties} \label{definitions}
We recall some facts about the Grothendieck filtration on $\lambda$-rings, in the context of character rings. A concise treatment of the basic facts about $\lambda$-rings can be found in \cite{atiyah-tall}. Let $G$ be a finite group, and let $\KK$ be a field of characteristic zero. The \textit{ring of virtual characters} (or character ring) $R_\KK(G)$ of $G$ is the augmented ring generated by irreducible characters of $G$ over $\KK$; the augmentation $\epsilon: R_\KK(G) \to \ZZ$ is the degree map. Note that, since $\KK$ has characteristic zero, representations up to isomorphism are determined by their characters. Thus $R(G)$ is also the Grothendieck ring on the category of $\KK G$-modules, and we use the terms "character" and "representation" interchangeably. For instance, if $\chi$ is a character of $G$, by "the $n$-th exterior power $\lambda^n(\chi)$ of $\chi$", we mean "the character associated to the $n$-th exterior power of the representation affording $\chi$". The maps $\{\lambda^n\}_n$ satisfy for all characters $\chi, \tau$:
\begin{enumerate}
	\item $\lambda^0(\chi) = 1$
	\item $\lambda^1(\chi) = \chi$
	\item $\lambda^k (\chi+\tau) = \sum_{i+j = k} \lambda^i(\chi)\lambda^j(\tau)$ \label{lambda_addition_formula}
\end{enumerate}
The addition formula above allows us to extend $\lambda^n$ to $R_\KK(G)$, by defining each $\lambda^n(-\chi)$ by the equation $\lambda^n(\chi+ (-\chi)) = 0$. We say that $R_\KK(G)$, together with the maps $\{ \lambda^n \}$, is a pre-$\lambda$-ring. Since the $\lambda$-operations also satisfy axioms \cite[\S 1 (12)-(14)]{atiyah-tall}, we see that $R_\KK(G)$ is a $\lambda$-ring. We define:
 \[
 	\lambda_T: \begin{cases} R(G) &\to 1 + T\cdot R(G)[[T]] \\
 		\rho &\mapsto 1 + \sum_{i = 1}^\infty \lambda^i (\rho)T^i
  	\end{cases}.
 \]
 Call $x$ a \textit{line element} if $\lambda_T(x) = 1 + xT$. Alternatively, $x$ is a line element whenever it is a one-dimensional representation of $G$.

\begin{Remark}
	In the terminology of \cite{atiyah-tall}, a ring with $\lambda$-operations satifying the first three axioms above is called a $\lambda$-ring, and the additional axioms make it a special $\lambda$-ring. These extra axioms describe in particular how $\lambda$-operations interact with the ring multiplication. As it turns out, they are equivalent to the so-called "splitting principle", stated below as \Cref{splittingprinciple}, and to which we refer in practice for calculations.
\end{Remark} 

For $x \in R_\KK(G)$ and $n\in \NN$, put
\[
	\gamma^n(x) = \lambda^n(x + n - 1),
\]
the $n$-th gamma operation. Let $I = \ker \epsilon$ be the augmentation ideal, and note that if $x \in I$ then $\gamma^n(x) \in I$. Let $\Gamma^n$ be the additive subgroup of $R_\KK(G)$ generated by the monomials
\[
	\gamma^{i_1}(x_1)\gamma^{i_2}(x_2)\cdots\gamma^{i_k}(x_k), \ x_i \in I, \  \sum_{j = 1}^{k} i_j \geq n.
\]
One can show that $\Gamma^0 = R(G)$, $\Gamma^1 = I$, and that each $\Gamma^n$ is a $\lambda$-ideal (see \cite[Prop. 4.1]{atiyah-tall}). Moreover, the $\Gamma$-filtration contains the $I$-adic filtration on $R_\KK(G)$, that is, $\Gamma^n \supseteq I^n$ for each $n$. These two filtrations contain the same topological information:
\begin{Proposition}[{\cite[Cor. 12.3]{atiyah-characters_cohomology}}] \label{gammatopology_coincides_idaictopology}
	The topology induced by the Grothendieck filtration coincides with the $I$-adic topology.
\end{Proposition}

Define the \textit{graded character ring of $G$} (with coefficients in $\KK$) as:
\[
	R^*_\KK(G) = \bigoplus_{i\geq 0}\Gamma^i/\Gamma^{i+1}.
\]
The definitions readily imply that $\Gamma^m\cdot\Gamma^n \subset \Gamma^{m+n}$, so this is indeed a graded ring. In the sequel, we simply write $R^*(G)$ whenever $\KK$ is clear from the context. Our aim is to compute examples of the graded ring $R^*(G)$ for some finite groups. \par
Determining generators for $R^*(G)$ is a completely straightforward process. For any $\rho \in R(G)$, let $C_n(\rho) = \gamma^n(\rho - \epsilon(\rho))$; we define the \textit{$n$-th algebraic Chern class} $c_n(\rho)$ of $\rho$ as the image of $C_n(\rho)$ by the quotient map $\Gamma^n(G) \to R^n(G)$.
Define 
 \[
 	c_T: \begin{cases} R(G) &\to 1 + T\cdot R^*(G)[[T]] \\
 		\rho &\mapsto 1 + \sum_{i = 1}^\infty c_i(\rho)T^i
  	\end{cases}.
 \]
We call $c_T(\rho)$ the \textit{total Chern class} of $\rho$. Note that if $x$ is a line element, then $c_T(x) = 1+c_1(x)T$.

\begin{Proposition}[{\cite[III.\S2]{fulton-lang}}] \label{chern_class_homomorphism}
	The total Chern class satisfies the axioms of a Chern class homomorphism as detailed in \cite[I.\S3]{fulton-lang}. In particular,
\begin{enumerate}[(i)]
	\item \label{chern_class_line_elements}If $\rho$ is the character of a representation of degree $n$, then $c_k(\rho) = 0$ for $k>n$.
	\item \label{chern_class_product} Whenever $\rho$ and $\sigma$ are line elements, we have $c_1(\rho\sigma) = c_1(\rho) + c_1(\sigma)$.
	\item \label{chern_class_addition}The map $c_T$ is a homomorphism, that is $c_T(x+y) = c_T(x)c_T(y)$. In particular 
	\[
		c_n(\rho+\sigma) = \sum_{i=0}^n c_i(\rho)c_{n-i}(\sigma).
	\]
\end{enumerate}
\end{Proposition}
Note that (\ref{chern_class_product}) is seen by remarking that
\[
	\gamma^1(\rho\sigma - 1) - \left(\gamma^1(\rho-1) + \gamma^1(\sigma -1)\right) = \gamma^1(\rho-1)\gamma^1(\sigma-1) \in \Gamma^2.
\]

Much as it is the case for $\lambda$-operations, whenever we need to compute the Chern class of a product, we rely on the splitting principle below.
\begin{Proposition}[{\cite[III.\S1]{fulton-lang}}]\label{splittingprinciple}
	Given representations $\rho_1, \cdots, \rho_k$ of $G$ dimensions $d_1,\cdots,d_k$ respectively, there exists a $\lambda$-ring extension $R'$ of $R(G)$ such that $\Gamma^nR'\cap R(G) = \Gamma^nR(G)$ and each $\rho_i = x_{i,1} + \cdots+x_{i,d_i}$ is the sum of $d_i$ line elements in $R'$.
\end{Proposition}
Thus for a character $\rho$ of degree $n$, we have, in the graded ring $gr^*R'$:
\[
	c_T(\rho) = c_T(x_1 + \cdots + x_n) = \prod_{i=1}^{n}(1+c_1(x_i)T).
\]
The graded ring $R^*(G)$ appears as a subring of $gr^*R'$, and we can recover $c_k(\rho)$ as the coefficient of $T^k$ in the above polynomial, that is, the symmetric polynomial of degree $k$ in the $n$ variables $c_1(x_1),\cdots,c_1(x_n)$. \par
As a first practical example of the splitting principle, consider the following computation. Recall that the determinant of a representation $\rho$ of $G$ of degree $n$ is defined as $\det(\rho) = \lambda^n(\rho)$. In particular, by the splitting principle, if we write $\rho = x_1 + \cdots + x_n$ then we have $\det \rho = \prod x_i$.

\begin{Lemma}\label{chern_class_determinant}
	For a representation $\rho$ of $G$, we have $c_1(\rho) = c_1(\det \rho)$.
\end{Lemma}
\begin{proof}
	Let $R'$ be an extension of $R(G)$ as in \Cref{splittingprinciple}. The ring $R'$ is a $\lambda$-ring, and (\ref{chern_class_product}) and (\ref{chern_class_addition}) of \Cref{chern_class_homomorphism} do apply in full generality. In $R'$, write $\rho = x_1+\cdots+x_n$ as a sum of line elements. Then, by \Cref{chern_class_homomorphism}(\ref{chern_class_addition}):
	\begin{align*}
		c_T(\rho) &= c_T(x_1 + \cdots + x_n) = \prod_{i = 1}^n c_T(x_i) \\
			&= \prod_{i=1}^n (1+c_1(x_i) T).
	\end{align*}
	The coefficient of $T$ is $c_1(\rho) = \sum_{i=1}^n c_1(x_i)$. On the other hand, by \Cref{chern_class_homomorphism}(\ref{chern_class_line_elements}):
	\[
		c_1(\det \rho) = c_1\left(\prod_{i=1}^n x_i\right) = \sum_{i=1}^n c_1(x_i) = c_1(\rho).
	\]
\end{proof}

Moreover, the splitting principle, together with property (\ref{chern_class_product}) in \Cref{chern_class_homomorphism}, imply that $c_k(\sigma\tau)$ is a polynomial in the Chern classes of $\sigma$ and $\tau$. As a direct consequence, we have:

\begin{Lemma} \label{R*G_generators}
	Let $\chi_1, \cdots, \chi_n$ be characters of representations of $G$ of degrees $d_1, \cdots, d_n$ respectively. If the $\chi_i$ generate $R(G)$ as a ring, then the classes $c_k(\chi_i)$ for $1 \leq k \leq d_i$ generate $R^*(G)$ as a ring.
\end{Lemma} 
\begin{proof}
	By definition, each $\Gamma^n$ is generated by products of Chern classes of virtual characters of $G$. The result follows from the above discussion.
\end{proof}

We conclude this section with the following improvement on \cite[Lem. 3.2]{guillot-minac}:
\begin{Proposition} \label{R*G_G-torsion}
	The graded piece $R^n(G)$ is $\vert G \vert$-torsion for $n>0$.
\end{Proposition}
\begin{proof}
	Consider the regular representation $\KK G$ of $G$, with character $\chi$, and let $\rho$ be any character of $G$. For any $g \in G$,
	\[
		\chi\cdot\rho(g) = (\epsilon(\rho)\cdot\vert G\vert )\delta_{1_G, g}.
	\]
	Pick a virtual character $\rho \in \Gamma^n$ for $n>0$, and write $\rho = \rho^+ - \rho^-$ with $\rho^+$, $\rho^- \in R^+(G)$ and $\epsilon(\rho^+) = \epsilon(\rho^-)$. Then  $\chi\cdot\rho^+ = \chi \cdot \rho^-$ and thus $\chi \cdot \rho = 0$. Looking modulo $\Gamma^{n+1}$, we obtain:
	\[
		0 = \chi\cdot\rho = (\chi - \vert G\vert)\rho + |G|\cdot\rho = |G|\cdot\rho \ (\Mod \Gamma^{n+1}),
	\]
	since $\chi-|G| \in I = \Gamma^1$.
\end{proof}

	\section{Computing from the definition: cyclic groups} \label{cyclicgroups}
	As introductory examples, we determine the graded character rings of some cyclic groups. In \Cref{gradedring_Cp}, we consider cyclic groups over an algebraically closed field: their graded character ring was computed in \cite{guillot-minac} and many of our subsequent examples will rely on it. For the sake of completeness, we reproduce here the calculation of Guillot and Min\'{a}\v{c}, which is an exercise in the definitions. \par
	In \Cref{prop_gradedringQQ}, we prove a surprising general result about graded character rings of $p$-groups over the rationals: the classical interplay between Adams operations and rationality (see \cite[Th. 13.29]{serre}), translates to a condition on the generators of $R^*_\QQ(G)$. We illustrate this statement in \Cref{R*GCp_QQ} with the computation of $R^*_\QQ(C_p)$, where $C_p$ is a cyclic group of prime order. This constitutes our only incursion outside the field of complex numbers; even the computation for cyclic groups of arbitrary order remains wide open over a general field. \par
	
	\begin{Proposition}[{\cite[Prop. 3.4]{guillot-minac}}]\label{gradedring_Cp}
		Let $C_N$ be the cyclic group of order $N$. Whenever $\KK$ is an algebraically closed field of characteristic prime to $N$,
		\[ 
			R^*_\KK(C_N)  = \frac{\ZZ\left[x\right]}{(Nx)} 
		\]
		with $x = c_1(\rho)$ for a one-dimensional representation $\rho$ of $C_N$ that generates $R_\KK(C_N)$.
	\end{Proposition}
	\begin{proof}
		Let $\rho$ be a generating character for $R(G)$. By \cref{R*G_generators}, its first chern class  $x = c_1(\rho)$ generates $R^*(G)$, and by \cref{R*G_G-torsion} we have $Nx = 0$. 
		It remains to show that there is no additional relation in $R^*(G)$, so suppose that $dx^n = 0$ for some $d$; that is, $d(\rho-1)^n \in \Gamma^{n+1}(G)$. We show that necessarily $N$ divides $d$. Note that since $G$ is cyclic, the augmentation ideal is $ I = (\rho - 1)$ and the Grothendieck filtration coincides with the $I$-adic filtration. Let $X = C_1(\rho)$, then the relation $dx^n = 0$ lifts to $dX^n = X^{n+1}P(X)$ in $R(G)$, which we can then lift to $\ZZ[X]$ as: 
		\[
			dX^n = P(X)X^{n+1} + Q(X)\left((X+1)^{N}-1\right).
		\]
		If $n > 1$, by considering the terms of degree $1$ on each side, we conclude that $Q(0) = 0$. We can then divide by $X$ and get a similar equation, with $dX^{n-1}$ on the left. We repeat this process until we reach an equation of the form
		\[
			dX = P(X)X^2 + Q(X)\left((X+1)^N - 1\right).
		\]
		By looking again at terms of degree $1$, we see that $d = NQ(0)$, which is what we wanted.
	\end{proof}
	
	Before we move on to more involved computations, here is a rather nice application of \cref{R*G_G-torsion} to graded character rings over the rationals. The proof of the following requires the use of Adams operations; they are $\lambda$-homomorphisms that exist on any $\lambda$-ring, whose precise definition and main properties are outlined in \cite[\S 5] {atiyah-tall}. For our purposes, it suffices to know that for a character of $G$, the $k$-th Adams operation is defined as $\psi^k( \chi(g) ) = \chi(g^k)$. 
	
	\begin{Proposition}\label{prop_gradedringQQ}
		Let $G$ be a $p$-group. Then $R^*_\QQ(G)$ is concentrated in degrees multiple of $(p-1)$. 
	\end{Proposition}	 
	\begin{proof}
		By \cite[Prop. 5.3]{atiyah-tall}, for $x \in \Gamma^n$, we have $\psi^k(x) = k^nx \ (\Mod \Gamma^{n+1})$. Moreover, by \cite[Th. 13.29]{serre}, over the rationals, $\psi^k(x) = x$ whenever $(|G|,k) = 1$. In particular, picking any $k \in (\ZZ/p\ZZ)^\times$ we have $(k^n-1)x \in \Gamma^{n+1}$. Since $x$ is $|G|$-torsion, we conclude that $x = 0$ whenever $(k^n-1) \neq 0 \ (\Mod p)$, that is, whenever $n$ is not a multiple of $(p-1)$. 
	\end{proof}

	A straightforward application of this result is the computation of $R^*_\QQ(G)$ for $G$ cyclic of prime order $p$.	
	
	\begin{Corollary}\label{R*GCp_QQ}
		\[
			R^*_\QQ(\ZZ/p\ZZ) = \frac{\ZZ[x]}{(px)}
		\]	
		with $x = c_{p-1}(\chi)$ where $\chi$ is the character of $\QQ[\ZZ/p\ZZ]$. 
	\end{Corollary}
	\begin{proof}
		Let $G = \ZZ/p\ZZ$. By \cite[Prop. 13.30 and Ex. 13.1]{serre}, the ring $R_\QQ(G)$ is generated by the characters of permutation representations of the subgroups of $G$.  The only two subgroups of $G$ are the trivial group $\{0\}$ and $G$ itself, so $R_\QQ(G)$ is generated by $\chi$, the regular representation. So $R^*_\QQ(G)$ is in turn generated by $c_i(\chi)$ for $1\leq i \leq p$; by \cref{prop_gradedringQQ}, it is generated by $c_{p-1}(\chi)$. It remains to show that this generator has additive order $p$ and is non-nilpotent. Consider the natural map $R_\QQ(G) \to R_\CC(G)$; it sends $\chi$ to $1+\rho+\cdots+\rho^{p-1}$, with $\rho$ a generating character of $R_\CC(G)$. The total Chern class of $1+\rho+\cdots+\rho^{p-1}$ is
		\[
			c_t\left(\sum_{i=0}^{p-1}\rho^i\right) = \prod_{i = 0}^{p-1}c_t(\rho^i) = \prod_{i = 0}^{p-1}(1+ic_1(\rho)),
		\]
		so $c_{p-1}(\rho) = (p-1)!\cdot c_1(\rho)$, which proves our claim using \cref{gradedring_Cp}.
	\end{proof}

The graded ring of a general cyclic group over the rationals is not known. In the sequel, unless mentioned otherwise, all graded character rings will be computed over the complex numbers.
	\section{The restriction homomorphism} \label{restrictions}
Graded character rings are computed in two steps: first, we identify a minimal set of generators for $R^*(G)$ using general information on the representation theory of $G$, and we determine possible relations in higher degree. The second step consists in showing that there are no extra relations in $R^*(G)$, and is usually much less straightforward. In the case of cyclic groups (see \Cref{gradedring_Cp}), we used an ad hoc method for this step; in this section we rely on the functoriality of $R^*(-)$ to look at restrictions of representations to subgroups of $G$. We rely on this technique, and on \Cref{gradedring_Cp}, to compute of $R^*(C_p^k)$ in \Cref{gradedrepring_Cpk}. We then turn to the dihedral groups $D_p$ for odd primes $p$ in \Cref{gradedring_Dp}, and to $D_4$ in \Cref{gradedring_D4}. In passing, we prove a Künneth formula for groups of coprime order.

\begin{Lemma}\label{R*G_functorial}
	The graded character ring $R^*(-)$ is functorial in $G$: a group homomorphism $\phi : G \to H$ induces a well-defined ring map $\phi^*: R^*(H) \to R^*(G)$, which sends each generator $c_n(\rho)$ to $c_n(\rho\circ\phi)$.\par
	 In particular, if $H$ is a subgroup of $G$, the restriction of representation $\res^{G}_{H}: R(G) \to R(H)$ induces a well-defined homomorphism of graded character rings, also denoted $\res^{G}_{H}: R^*(G) \to R^*(H)$, with
	\[
		\res^{G}_{H}c_n(x) = c_n(\res^{G}_{H}(x))
	\] 
	for all $x\in R(G), n \in \NN$.
\end{Lemma}
\begin{proof}
	This is clear.
\end{proof}

The second consequence of \Cref{R*G_functorial} is to reduce the computation of graded character rings of abelian groups to that of $R^*(G)$ for $p$-groups:
\begin{Corollary}\label{GH_coprime} Let $G$ and $H$ be groups with coprime order. Then 
	\[
		R^*_\CC(G\times H) = R^*_\CC(G) \otimes_\ZZ R^*_\CC(H)		
	\]
\end{Corollary}
\begin{proof}
	Let $\pi_G, \pi_H: G\times H \to G,H$ be the projection maps. By \cite[Th. 3.10]{serre}, for any complex irreducible character $\rho$ of $G\times H$, there are irreducible characters $\sigma_G,\sigma_H$ of $G, H$ respectively such that $\rho = (\sigma_G \circ\pi_G) \cdot (\sigma_H\circ\pi_H$). Let $\rho_G = \sigma_G\circ \pi_G$ and $\rho_H = \sigma_H\circ\pi_H$, then $R^*_\CC(G\times H)$ is generated by classes of the form $c_n(\rho_G\cdot\rho_H)$, which can be written as polynomials in the Chern classes of $\rho_G, \rho_H$. In other words, the projection maps $\pi_G, \pi_H$ induce a surjective homomorphism 
	\[
		\pi_G^*\otimes\pi_H^* : R^*_\CC(G)\otimes R^*_\CC(H) \to R^*_\CC(G\times H).
	\]
	Moreover, applying \Cref{R*G_G-torsion} to the case where $|G|, |H|$ are coprime, we have that
	\[
		\bigoplus_{i+j = n}\left(R^i(G)\otimes R^j(H)\right) \cong R^n_\CC(G) \oplus R^n_\CC(H)
	\]
	for any $n\geq 1$. In particular, the surjection above decomposes as $\bigoplus(\pi_G^*)^{n} \otimes (\pi_H^*)^n: R^n_\CC(G)\oplus R^n_\CC(H) \to R^*_\CC(G\times H)$. The inclusions $\iota_G,\iota_H: G,H \to G\times H$ induce a two-sided inverse $(\iota^*_G,\iota^*_H) = \bigoplus_{n} (\iota_G^n, \iota_H^n): R^n_\CC(G\times H) \to R^n_\CC(G)\oplus R^n_\CC(H)$ to this surjection.
\end{proof}

We now use \Cref{R*G_functorial} to compute of the graded character rings of elementary abelian groups. Let $p$ be a prime number, and let $C_p$ be the cyclic group of order $p$, with a choice of generator $g$. Recall that we fixed $\KK = \CC$.
	\begin{Proposition} \label{gradedrepring_Cpk}
		Let $G = C_p^k$. Then
			\[
				R^*(G) = \frac{\ZZ[x_1,\cdots,x_k]}{(px_i,x_i^px_j = x_ix_j^p)}.
			\]
		with $x_i = c_1(\rho_i)$ where $\rho_i$ restricts to a nontrivial one-dimensional representation of the $i$-th factor $C_p$.
	\end{Proposition}
	\begin{proof}
		 Denote by $g_i$ the element $(1,\cdots,1,g,1,\cdots,1)$ with $g$ in $i$-th position, so that $G$ is generated by $g_1,\cdots,g_k$. Let $\omega = \exp^{2i\pi/p}$, and let $\rho_i$ be the representation of $G$ defined by $\rho_i: g_j \mapsto \omega^{\delta_{ij}}$. The $\rho_i$'s generate $R(G)$, so the elements $x_i := c_1(\rho_i)$ generate $R^*(G)$. Note that $\rho_i^p = 1$ for all $i$, so $px_i = 0$ by \Cref{chern_class_homomorphism}(\ref{chern_class_product}).\\
		 The relation $x_i^px_j = x_ix_j^p$ is obtained as follows: let $X_i$ be the standard lift $\rho_i - 1$ of $x_i$ to $R(G)$. Then $(X_i+1)^p = 1$, so that
			\begin{align*}
				X_i^p &= -\sum_{l=1}^{p-1}\binom{p}{l}X_i^l = X_i\left(-\sum_{l=0}^{p-2}\binom{p}{l+1}X_i^l\right) \\
					&= pX_i(-1+\phi(X_i)),
			\end{align*}
			where $\phi(T) \in \ZZ[T]$ has no constant term. For any $i,j$, write
			\begin{align*}
			 	 X_i^pX_j\left( -1 + \phi(X_j) \right) &= pX_iX_j\left(-1+\phi(X_i)\right)\left(-1+\phi(X_j)\right) \\
			 	 	&= X_iX_j^p\left(-1+\phi(X_i)\right).
			\end{align*}
			In $R^{p+1}(G)$, this is:
			\[
				x_i^px_j = x_ix_j^p,
			\]
			and the generators of $R^*(G)$ satisfy all the required relations. \\
			Let us show that there are no extra relations: the graded piece of rank $l$ is generated by monomials of the form:
			\[
				x_1^{s_1}\cdots x_k^{s_k}, \ \ \ \  \sum_{i=1}^{k} s_i = l.
			\]
			Let $S_l \subset \ZZ^k_{\geq 0}$ be the set of multi-indices $(s_1,\cdots,s_k)$ such that $\sum s_i = l$ and only the first nonzero coordinate of each $s\in S_l$ is (possibly) greater than $p-1$. We must show that the monomials $x^s = x_1^{s_1}\cdots x_k^{s_k}$ are linearly independent. Consider a zero linear combination:
			\begin{align}
				\sum_{s\in S_l}a_sx^s = 0 \label{Cpk_linearcombination}
			\end{align}
			and let $\psi: R^*(C_p^k) \rightarrow R^*(C_p)$ be the restriction to the cyclic group generated by the product $g_1^{t_1}\cdots g_k^{t_k}$ for some $0\leq t_j \leq p-1$. Then $\psi(x_j) = t_j\cdot z$, where $z$ is the standard one-degree generator of $R^*(C_p)$, and \Cref{Cpk_linearcombination} becomes:
			\begin{equation*}
				\sum_{s\in S_l}a_st_1^{s_1}\cdots t_k^{s_k}z^l = 0,
			\end{equation*}
			that is,
			\begin{equation}
				\sum_{s\in S_l}a_st_1^{s_1}\cdots t_k^{s_k} = 0 \in \Fp, \label{Cpk_linearcombinationFp}
			\end{equation}
			for all possible strings $(t_1,\cdots,t_k)$ with $0\leq t_j \leq p-1$. In particular, grouping terms by powers of $t_k$ in \Cref{Cpk_linearcombinationFp}, we get:
			\begin{align}
				\begin{dcases} \ \ \ a_{(0,\cdots,0,l)}t_k^{l} = 0 &\text{ when } t_1 = \cdots = t_{k-1} = 0 \\
					\ \ \ \sum_{t=0}^{p-1}\left(\sum_{s\in S_{l-i}} b_st_1^{s_1}\cdots t_{k-1}^{s_{k-1}} \right) t_k^i = 0 &\text{otherwise.}	
				\end{dcases}				
			\end{align}
			This implies that the coefficient of $x_k^l$ in \Cref{Cpk_linearcombination} is zero; more generally, the second equation must be true for all values of $t_k$, from $0$ to $p-1$. In other words, the $\left( \sum b_st_1^{s_1}\cdots t_{k-1}^{s_{k-1}} \right)$ are the entries of a vector in the kernel of the Vandermonde matrix $(t_k^i)_{t_k = 1,\cdots,p-1}^{i = 1,\cdots,p-1}$, which is invertible in $\Fp$. Therefore 
	\[
		\sum_{s\in S_{l-i}} b_st_1^{s_1}\cdots t_{k-1}^{s_{k-1}} = 0
	\]		
	for all combinations $(t_1,\cdots,t_{k-1})$. An immediate induction shows that we must have each $a_s = 0$, so the monomials $\lbrace x^s\rbrace_{s\in S_l}$ are linearly independent.
	\end{proof}
	Note that the relations between the generators of $R^*((C_p)^k)$ appear in degree $p+1$, so the degree of relations goes to $\infty$ as $p \to \infty$. In \cref{quillen} we shed light on this phenomenon, via a general result about the minimal degree of relations in a $p$-group. 

	We now turn to our first non-abelian group, for which the computation combines the restriction map with some basic Chern class algebra. Let $p$ be an odd prime and consider the dihedral group:
	\[
	D_p = \left\langle \tau, \sigma \vert \tau^2 = \sigma^p = 1, \tau\sigma\tau = \sigma^{-1} \right\rangle.
	\]	
	There are $(p+1)/2$ irreducible representations of $D_p$: 
	\begin{itemize}
		\item Two representations of degree 1, the trivial representation 1 and the signature $\varepsilon$ which sends elements of the form $\sigma^j$ to 1, and elements of the form $\tau\sigma^j$ to -1.
		\item And $(p-1)/2$ representations $\chi_1,  \cdots \chi_{(p-1)/2}$ of degree 2:
			\begin{align*}
			\chi_k(\sigma^j) &= \begin{pmatrix} e^{\frac{2ikj\pi}{p}} & 0 \\
				0 &  e^{-\frac{2ikj\pi}{p}}\end{pmatrix} \\
			\chi_k(\tau) &= \begin{pmatrix} 0 & 1 \\ 1 & 0 \end{pmatrix} \begin{matrix} \vphantom{0}\\ \vphantom{1}.\end{matrix}
			\end{align*}
	\end{itemize}
	The characters of these generate the ring $R(D_p)$. For convenience, define $\chi_0 = 1+\varepsilon$; we have the following relations:
	\begin{align}
		\varepsilon^2 &= 1  \label{Dp_order_epsilon} \\
		\varepsilon\cdot\chi_k &= \chi_k \\
		\chi_k\cdot\chi_l &= \chi_{k+l} + \chi_{k-l}. \label{Dp_chikchil}
	\end{align}
	\begin{Proposition} \label{gradedring_Dp}
		Let $x = c_1(\chi_1)$ and $y = c_2(\chi_1)$, then
		\[
		R^*(D_p) = \frac{\ZZ\left[x,y\right]}{(2x,py,xy)}.
		\]
	\end{Proposition}	
	\begin{proof}
		Let $x$ and $y$ be as above; note that \Cref{chern_class_determinant} implies that $x = c_1(\chi_1) = c_1(\chi_k) = c_1(\mathrm{det} \ \chi_k) = c_1(\varepsilon)$ for any $k$, and $0 = c_1(\varepsilon^2) = 2c_1(\varepsilon) = 2x$. For the other relations, we use \Cref{Dp_chikchil} above and apply the total Chern class $c_t$ to both sides:
		\begin{equation}
			c_t(\chi_k\chi_l) = c_t(\chi_{k+l})c_t(\chi_{k-l}). \label{Dp_totalclass}
		\end{equation}
		Let $y_i = c_2(\chi_i)$. Expand the right-hand side:
		\begin{align*}
			c_t(\chi_{k+l})c_t(\chi_{k-l}) &= (1+xT+y_{k+l}T^2)(1+xT+y_{k-l}T^2) \\
				&= 1 + 2xT + (x^2+ y_{k+l} + y_{k-l})T^2 + (xy_{k+l} + xy_{k-l})T^3 + y_{k+l}y_{k-l}T^4. \numberthis \label{Dp_ctrhs}
		\end{align*}
		For the left-hand side, we use the splitting principle (\Cref{splittingprinciple}): in some extension of $R(D_p)$, we can write $\chi_k = \rho_1 + \rho_2$ and $\chi_l = \eta_1 + \eta_2$ with $\rho_i$, $\eta_i$ of dimension 1, in a way that is compatible with the $\Gamma$-filtration. Then:
		\begin{align*}
			c_t(\chi_k\chi_l) =& c_t\left((\rho_1+\rho_2)(\eta_1+\eta_2)\right) \\
				=& c_t(\rho_1\eta_1)c_t(\rho_1\eta_2)c_t(\rho_2\eta_1)c_t(\rho_2\eta_2) \\
				=& (1+(c_1(\rho_1)+c_1(\eta_1))T)\cdot(1+(c_1(\rho_1)+c_1(\eta_2))T)\\ 
					&\cdot(1+(c_1(\rho_2)+c_1(\eta_1))T)\cdot(1+(c_1(\rho_2)+c_1(\eta_2))T).
		\end{align*}
		Now, let $s_1, s_2$ (resp. $t_1, t_2$) be the first and second symmetric polynomials in $(\rho_1, \rho_2)$ (resp. ($\eta_1,\eta_2)$). Then $c_i(\chi_k) = s_i$ and $c_i(\chi_l) = t_i$. The last equality can be rewritten:
		\begin{align*}
			c_t(\chi_k\chi_l) = &1+ 2(s_1+t_1)T + (t_1^2+s_1^2+3s_1t_1+2t_2+2s_2)T^2 \\
				&+(s_1^2t_1+s_1t_1^2+2s_1s_2+2t_1t_2+2s_1t_2 + 2s_2t_1)T^3 \\
				&+ (t_2^2 + s_2^2 + s_1s_2t_1 + s_1t_1t_2 + s_1^2t_2 + s_2t_1^2 -2s_2t_2)T^4.
		\end{align*}
		We replace $s_1 = t_1 = x$ and eliminate all occurrences of $2x$ to obtain
		\begin{equation}
			c_t(\chi_k\chi_l) = 1 + (x^2+2y_k+2y_l)T^2 + (y_k-y_l)^2T^4. \label{Dp_ctlhs}
		\end{equation}
		Comparing coefficients in \Cref{Dp_ctrhs} and \Cref{Dp_ctlhs}, we obtain:
		\begin{align}
			y_{k+l}+y_{k-l} &= 2(y_k + y_l) \label{Dp_relationsum} \\
			xy_{k+l} + xy_{k-l} &= 0 \label{Dp_relationproduct} \\
			y_{k+l}y_{k-l} &= (y_k-y_l)^2 .
		\end{align}
		First look at \Cref{Dp_relationproduct} with $k=l$. Note that $y_0 = c_2(1+\varepsilon) = 0$, and thus \Cref{Dp_relationproduct} yields $x\cdot y_{2k} = 0$ for all $k$, which is equivalent to $x\cdot y_k = 0$ for all $k$ since indices are understood modulo the odd prime $p$.\\
		We then show that $py_k = 0$ for all $k$. Recall that $R^*(D_p)$ is $2p$-torsion, and consider \Cref{Dp_relationsum} with $k=l$. Multiplying by $p$, we obtain $py_{2k} = 0$ for all $k$. Again, this implies that $py_k = 0$ for all $k$. \\
		Finally, consider \Cref{Dp_relationsum} with $l = k$ and $l = k+1$. This gives:
		\begin{align*}
			y_{2k} &= 4y_k \\
			y_{2k+1} &= 2(y_k+y_{k+1}) - y_1.
		\end{align*}
		Together, these two relations imply that all $y_k$'s are multiples of $y_1 =: y$.\\		
		It remains to show that these are the only relations in $R^*(D_p)$, that is, $x$ and $y$ are not nilpotent, and there is no extra dependency relation between them. Restricting $x$ to $C_2$ and $y$ to $C_p$ shows none of the generators are nilpotent, while restricting both $x$ and $y$ to $C_2$ eliminates any extra possible relation.
	\end{proof}
	
	A similar argument gives $R^*(D_4)$, where $D_4$ is the dihedral group $D_4$ of order 8. Note that $R^*(D_4)\otimes \FF_2$ is already known and was computed in \cite[Prop. 3.12]{guillot-minac}. It has four nontrivial irreducible representations:
	 \begin{itemize}
	 	\item In degree 1, the representations $\rho: r \mapsto -1, \ s\mapsto 1$ and $\eta: r \mapsto 1, s\mapsto -1$ and their product $\rho\eta$,
	 	\item And in degree 2, the representation $\Delta$, which sends $s$ to $\begin{pmatrix} 1 & 0 \\ 0 & -1 \end{pmatrix}$ and $r$ to $\begin{pmatrix} 0 & -1 \\ 1 & 0\end{pmatrix}$,
	 \end{itemize}
	with relations:
	\begin{align}
		\rho^2 &= \eta^2 = 1 \\
		\rho\Delta &= \eta\Delta = \Delta  \\
		\Delta^2 &= 1+\rho+\eta+\rho\eta. \label{D4_relationdelta}
	\end{align}
	\begin{Proposition}\label{gradedring_D4}
	Let $c_1(\rho) = x$; $c_1(\eta) = y$ and $c_2(\Delta) = b$. Then 
		\[
		R^*(D_4) = \frac{\ZZ\left[x,y,b\right]}{(2x,2y,4b,xy,xb-yb)}.
		\]
	\end{Proposition}
	\begin{proof}
	Note that $c_1(\rho\eta) = x+y$ and $c_1(\Delta) = c_1(\det\Delta) = c_1(\rho\eta)$. So the graded ring is indeed generated by $x,y,b$. We have $2x = 2y = 0$ from the relations above; and, letting $X,Y,B$ being the standard lifts $C_1(\rho), C_1(\eta), C_2(\Delta)$ of $x,y$ and $b$ to $R(D_4)$, we compute that $XB = YB = XY$. So $XY \in \Gamma^3$, thus $xy = 0$ and $xb = yb$. Finally, applying the total Chern class to \Cref{D4_relationdelta} yields the equation:
	\[
		5c_1(\Delta)^2 + 4b = x^2+3xy+y^2 = x^2+ y^2 = (x+y)^2
	\]
	and since $c_1(\Delta) = x+y$, we obtain $4b = 0$. \\
	To see these are the only relations, we use the computation of $R^*(D_4)\otimes\FF_2 = \frac{\ZZ[x,y,b]}{(xy,xb-yb)}$ from \cite{guillot-minac}: tensoring with $\FF_2$ shows that none of $x,y,b$ is nilpotent and that there are no extra relations between the genrerators. Finally, restriction to $C_4 = \langle r \rangle$ shows that any power $b^i$ of $b$ has additive order $4$. 
	\end{proof}		
	
	\section{Universal enveloping algebras} \label{quillen}
The aim of this section is to construct, for any abelian $p$-group $G$, a map:
\[
	R^*(G)\otimes \Fp \to gr_\bullet (\Fp G),
\]
where $gr_\bullet \Fp G$ is the graded ring associated to the filtration of the group ring $\Fp G$ by powers of its augmentation ideal. To this effect, we apply the main result of \cite{quillen}: fix a prime $p$, and let $\{G_n\}$ denote the lower central series of $G$, defined by $G_1 = G$ and $ G_{n+1} = (G_n,G) $. Consider the sequence $\{D_n\}$, where $D_n$ is the $n$-th mod $p$ dimension subgroup of $G$:
\[
D_n := \prod_{ip^s \geq n}G_i^{p^s}.
\]

Then $\{D_{n}\}$ is a $p$-filtration of $G$, that is, it satisfies:
\begin{itemize}
	\item $(D_r, D_s) \subseteq D_{r+s}$
	\item $x\in D_r \implies x^p \in D_{pr}$ for all $r$.
\end{itemize}
Moreover, $\{D_{n}\}$ is the fastest descending $p$-filtration of $G$ (see \cite[\S 11.1]{dixon}). Set $L_\bullet(G) = \bigoplus D_n / D_{n+1}$, then $L_\bullet(G)$ is a $p$-restricted graded Lie algebra over $\Fp$.
On the other hand, if $I$ denotes the augmentation ideal of the group ring $\Fp G$, then
\[
  F_n := \{x \in G \ | \ x-1 \in I^n \}
\]
is also a $p$-filtration of $G$, thus $F_n \supset D_n$ and there is a map of Lie algebras:
\[
	\psi: \begin{cases} L_\bullet (G)&\to gr_\bullet (\Fp G) \\
		g \ (\mathrm{mod} \ F^n) &\mapsto (g-1) \ (\mathrm{mod} \ I^n). \end{cases}
\]

\begin{Theorem}[{\cite[\S 1]{quillen}}] 
	The homomorphism $\widehat{\psi}$ from $gr_\bullet (\mathbb{F}_pG)$ to the universal enveloping algebra $U(L_\bullet(G))$ induced by $\psi$  is an isomorphism. 
\end{Theorem}

Now suppose $G$ is an abelian $p$-group of the form $C_{p^{i_1}}\times \dots \times C_{p^{i_m}}$. Then $D_{n} = G^{p^i}$ for $p^i$ the smallest power of $p$ such that $p^i \geq n$. Thus
\[
L_n(G) \cong \begin{cases} \{1\}, &n \neq p^i \\
  C_p\times\cdots \times C_p, &n = p^i.
 \end{cases}
\]

\begin{Lemma}
	If $G$ is abelian then $R(G)\otimes \mathbb{F}_p \cong \mathbb{F}_pG$ through an isomorphism that sends the Grothendieck filtration $\{\Gamma_p^n\}_n$ induced on $R(G)\otimes\Fp$ to the $I$-adic filtration on $\Fp G$.
\end{Lemma} 
\begin{proof}
	Write $G$ as a product of cyclic groups. The isomorphism that sends each cyclic group generator $g$ to the character $\rho_g: g \mapsto e^{2\pi i /|g|}$, sends $I \subset \Fp$ to $\Gamma_p^1$. Since every irreducible character of $G$ has dimension 1, the filtration $\{\Gamma_p^n\}_n$ coincides with the $\Gamma_p^1$-adic filtration.
\end{proof}

So $gr_\bullet(R(G) \otimes \mathbb{F}_p) \cong gr_\bullet(\mathbb{F}_pG) \cong U(L_\bullet(G))$ with universal map
\[
  h: \begin{cases} L_\bullet(G) &\to gr_\bullet(R(G)\otimes\mathbb{F}_p) \\
    g &\mapsto C_1(\rho_g) \ (\Mod \Gamma_p^2)
    \end{cases},
\]
and there is a map $\phi$ of algebras induced by $L_\bullet(G) \to gr_\bullet(\Fp G)$:
\[
	\phi: gr_\bullet(R(G)\otimes \Fp) \to gr_\bullet(\mathbb{F}_pG).
\]
On the other hand, the map $R(G) \to R(G)\otimes \Fp$ preserves the $\Gamma$-filtration and induces a maps $R^*(G) \to gr_\bullet(R(G)\otimes \Fp)$, and thus a map $R^*(G)\otimes \Fp \to gr_\bullet(R(G)\otimes\Fp)$. Composing this latter map with $\phi$, we obtain a map 
\[
	R^*(G)\otimes \Fp \to gr_\bullet(\Fp G)
\]
satisfying $\phi(c_1(\rho_g)) = g-1$.

A straightforward corollary of this is the following:
\begin{Theorem}\label{theorem_quillenrelations}
	Let $G = C_{p^{i_1}}\times\cdots\times C_{p^{i_n}}$, and let $\rho_k$ be the generating character of $R(C_{p^{i_k}})$ sending a generator $g_k$ of $C_{p^{i_k}}$ to $e^{2i\pi/p^{i_k}}$. Then there is a well-defined homomorphism:
	\[
		R^*(G)\otimes\Fp \to \frac{\Fp[u_1,\cdots, u_n]}{(u_1^{p^{i_1}}, \cdots, u_n^{p^{i_n}} )} 
	\]
	sending $c_1(\rho_k)$ to $u_k$.  
\end{Theorem}

Although we do not directly refer to it in the sequel, \Cref{theorem_quillenrelations} proves useful when "guessing" relations in $R^*(G)$, as illustrated by the graded character rings of abelian 2-groups: in \Cref{Z4Z4}, we show that
\[
			R^*(C_4 \times C_4) = \frac{\ZZ[x,y]}{(4x,4y,2x^2y+2xy^2, x^4y^2 - x^2y^4)}
\]
with $x = c_1(\rho_{(1,0)})$, $y = c_1(\rho_{(0,1)})$. By \Cref{theorem_quillenrelations}, modulo 2, nontrivial relations must involve $x^4$ or $y^4$. Since one can easily rule out relations of the form $x^4, y^4 = 0$ by restriction to $C_4$, we know that any extra relation will occur in degree 5 or more. Here, it occurs in degree 6.
Again, in \Cref{Z4Z2}, we show
\[
	R^*(C_4 \times C_2) = \frac{\ZZ[x,y]}{(4x,2y,xy^3+x^2y^2)}
\]
with $x = c_1(\rho_{(1,0)})$, $y = c_1(\rho_{(0,1)})$. We know by \Cref{theorem_quillenrelations} that any nontrivial relation modulo $2$ must involve $x^2$ or $y^4$.

	 \section{Continuity of characters} \label{continuity}

In the sequel, we view $R(G)$ as a topological ring, with the topology induced by the filtration $\lbrace \Gamma^n \rbrace$; that is, a subset $U \subseteq R(G)$ is open if for any $x\in U$ there is a $t$ such that $x + \Gamma^t \subseteq U$.
	If $G$ has exponent $m$, then each conjugacy class representative $g \in G$ gives rise to a ring morphism:
	\[
		\phi_g : \begin{cases} R(G) \rightarrow \ZZ[\mu_m] \\
			\ \ \rho \mapsto \chi_\rho(g) \end{cases},
	\]
	where $\mu_m$ is a choice of primitive $m$-th root of unity. We are interested in continuity and density questions with respect to $p$-adic topologies on $\ZZ$. Note that, to make any kind of rigorous statement, we need to fix an extension of the $p$-adic valuation to $\ZZ[\mu_m]$. However, we are primarily interested in the case where $m$ is a power of $p$; in that case, as is well-known, there is only one such extension. In particular, \Cref{proposition_pgroupscontinuous} states that whever $G$ is a $p$-group, all evaluation morphisms are continuous. We apply this result in \Cref{Q8_gradedring} to the computation of $R^*(Q_8)$. \par
	Suppose we are given additive groups $\widetilde{\Gamma}^n \subseteq \Gamma^n$ $(n\geq 1)$ such that:
	\begin{enumerate}[A.]
		\item $\widetilde{\Gamma}^{n+1} \subseteq \widetilde{\Gamma}^n$ \label{admissiblefiltration1}
		\item $\Gamma^n = \widetilde{\Gamma}^n + \Gamma^{n+1}$ \label{admissiblefiltration2}
	\end{enumerate}
	(think of $\widetilde{\Gamma}^n$ as an approximation of $\Gamma^n$). Then by an immediate induction:
	\begin{Lemma}
		For all $k\in\mathbb{N}$,
			\[
				\Gamma^n = \widetilde{\Gamma}^n + \Gamma^{n+k}
			\] \qed
	\end{Lemma}
	
	Call $\{\widetilde{\Gamma}^n \}_n$ an \textit{admissible approximation} for $\{\Gamma^n\}_n$ if it satisfies conditions (\ref{admissiblefiltration1}) and (\ref{admissiblefiltration2}).	
	\begin{Remark}
		Whenever $\{\widetilde{\Gamma}^n\}$ is an admissible approximation, each $\widetilde{\Gamma}^n$ is dense in $\Gamma^n$ for the $\Gamma$-topology.
	\end{Remark}

	\begin{Proposition}\label{evaluation_approximation}
		Let $p$ be a prime number, and suppose the evaluation morphisms 
		\[
			\phi_1,\cdots \phi_k : R(G) \mapsto \ZZ[\mu_m]		
		\]
		are continuous with respect to the topology induced by the filtration $\lbrace \Gamma^n \rbrace$ on $R(G)$, and the $p$-adic topology on $\ZZ[\mu_m]$. Then for all $x \in \Gamma^n$, and for all $M>0$, there is an element $\widetilde{x} \in \widetilde{\Gamma}^n$ such that for all $i = 1,\cdots,k$:
		\[
			\begin{cases} v_p(\phi_i(\widetilde{x})) =  v_p(\phi_i(x)) \text{ whenever } v_p(\phi_i(x)) < +\infty \\
			v_p(\phi_i(\widetilde{x}))>M \text{ whenever } \phi_i(x) = +\infty  			
			\end{cases}
		\]
	\end{Proposition} 
	\begin{proof}
		
		Let $x \in \Gamma^n$, $M>0$. Since all the $\phi_i$ are continuous with respect to the $p$-adic topology, there exists $N$ such that for all $j$ and for all $y \in \Gamma^N$ we have
		\[
			v_p(\phi_j(y)) > \max\left( \max_{v_p(\phi_i(x)) < \infty} \ v_p(\phi_i(x)), \ M \right).
		\]
		We can then write $x = \widetilde{x}+r$ with $\widetilde{x}\in\widetilde{\Gamma}^n$ and $r \in \Gamma^N$.
	\end{proof}
	
	\begin{Proposition}\label{proposition_pgroupscontinuous}
		Let $G$ be a $p$-group. Then the morphisms $\phi_g$, for $g \in G$, are all continuous with respect to the $p$-adic topology on $\ZZ[\mu_m]$.
	\end{Proposition}
	\begin{proof}
		Fix an element $g \in G$ and let $|G| =: p^n$. By \Cref{gammatopology_coincides_idaictopology}, it suffices to show that $\phi_g$ is continuous with respect to the $I$-adic topology on the left. We show that for any irreducible character $\rho$ of $G$,
		\[
			v_p\left(\phi_g\left(\rho - \varepsilon\left(\rho\right)\right)\right) > 0,
		\]
		which implies continuity. Since $G$ is a $p$-group, every character is a sum of $p^n$-th roots of unity, so
		\begin{align*}
			\phi_g(\rho-\varepsilon(\rho)) &= \rho(g) - \varepsilon(\rho)	\\
				&= \sum_{l = 1}^{\varepsilon(\rho)}(\mu_{p^n}^{i_l} - 1),
		\end{align*}
		and each $(\mu_{p^n}^{i_l} - 1)$ has positive $p$-valuation.
	\end{proof}
			
		The continuity method allows us to solve questions of torsion and nilpotency in $R^*(G)$: if some element $x \in R(G)$ is contained in $\Gamma^M$ with $M$ large, then $\phi_g(x)$ must be divisible by a large power of $p$. Here is a concrete example: let $G = Q_8 = \langle \ i,j,k \ \vert \ i^2 = j^2 = k^2 = ijk \ \rangle$ be the quaternion group of order $8$.
	\begin{Theorem}	 \label{Q8_gradedring}
		\[
		R^*(Q_8) = \frac{\ZZ[x,y,u]}{(2x,2y,8u,x^2,y^2, xy-4u)}
		\]
		where $\vert x \vert = \vert y \vert = 1$ and $\vert u \vert = 2$, with explicit generators as described in \Cref{Q8_relations}.
	\end{Theorem} 
	We will use continuity of characters to show that the additive order of $u$ is $8$ .\\
	The group $Q_8$ has 5 conjugacy classes: $\lbrace 1\rbrace$, $\lbrace -1\rbrace$, $\lbrace \pm i\rbrace$, $\lbrace \pm j\rbrace$, $\lbrace \pm k\rbrace$ so 5 irreducible representations on $\CC$. They are as follows:
	\begin{itemize}
		\item In dimension 1, the trivial representation, and the characters $\rho_1: \begin{cases} i \mapsto 1 \\ j\mapsto -1 \end{cases}$, $\rho_2 = -\rho_1$ and $\rho_3 = \rho_1\rho_2$,
		\item and in dimension 2, the representation $\Delta$:
		\begin{equation*}
		 \Delta(i) = \begin{pmatrix} i & 0 \\ 0 & -i \end{pmatrix}, \ \ \ \Delta(j) = \begin{pmatrix} 0 & -1 \\ 1 & 0 \end{pmatrix}, \ \ \ \Delta(k) = \begin{pmatrix} 0 & -i \\ -i & 0 \end{pmatrix} \begin{matrix} \vphantom{1} \\ \vphantom{0}. \end{matrix}
		\end{equation*}
	\end{itemize}
	These representations satisfy the relations
	\begin{align}
		\rho_i^2 &= 1	\\
		\rho_3 &= \rho_1\rho_2 \\
		\Delta\rho_i &= \Delta  \label{Q8_relationdelta} \\
		\Delta^2 &= 1 + \rho_1 + \rho_2 + \rho_3.
	\end{align}
	
	Let us first take a look at the generators and relations of $Q_8$:
	\begin{Lemma} \label{Q8_relations}
		The graded ring $R^*(Q_8)$ is generated by the elements 
		\[
			x := c_1(\rho_1), \ \ \ \ y := c_1(\rho_2), \ \ \ \ u := c_2(\Delta),
		\]
		which satisfy the relations in \Cref{Q8_gradedring}, that is:
		\[
			2x = 2y = 8u = 0 ,  \ \ \ \ x^2 = y^2 = 0, \ \ \ \ xy = 4u.
		\]
	\end{Lemma}
	\begin{proof}
		First, we eliminate the redundant generators: $c_1(\rho_3) = c_1(\rho_1\rho_2) = c_1(\rho_1) + c_1(\rho_2)$, and $c_1(\Delta) = c_1(\det (\Delta) ) = c_1(1) = 0$. So $R^*(Q_8)$ is indeed generated by $x,y$ and $u$. \\
		Now since $\rho_1^2 = \rho_2^2 = 1$, we have $2x = 2y = 0$, and the order of $Q_8$ kills $R^*(Q_8)$ so $8u = 0$. For the relations in degree 2, we apply the total Chern class to \Cref{Q8_relationdelta}, splitting the 2-dimensional representation $\Delta$ into $\sigma_1+\sigma_2$. On the left-hand side we have:
		\begin{align*}
			c_t(\Delta\rho_i) =& c_t(\sigma_1\rho_i)c_t(\sigma_2\rho_i) \\
				=& 1 + \left[ c_1(\sigma_1) + c_1(\sigma_2) + 2c_1(\rho_i) \right]T \\
					&+ \left[ c_1(\sigma_1)c_1(\sigma_2) + c_1(\sigma_1)c_1(\rho_i) + c_1(\sigma_2)c_1(\rho_i) + c_1(\rho_i)^2 \right]T^2
		\end{align*}
		While on the right-hand side:
		\begin{align*}
			c_t(\Delta) = 1 + \left[c_1(\sigma_1)+c_1(\sigma_2)\right]T + \left[c_1(\sigma_1)c_1(\sigma_2)\right]T^2.
		\end{align*}
		In degree 2, this yields the relation:
		\begin{align*}
			c_1(\rho_i)(c_1(\sigma_1)+c_1(\sigma_2)) + c_1(\rho_i)^2 = 0.
		\end{align*}
		Bearing in mind that $c_1(\sigma_1)+c_1(\sigma_2) = c_1(\Delta) = 0$, we obtain $x^2 = y^2 = 0$.
		The relation $xy = 4u$ is obtained by applying $c_t$ to the relation $\Delta^2 = 1 + \rho_1 + \rho_2 + \rho_1\rho_2$ and identifying the terms in degree 2, which yields
		\begin{align*}
			5c_1(\Delta)^2+4u = x^2 + y^2 + 3xy,
		\end{align*}
		that is, $4u = xy$.
	\end{proof}
	
	What remains to be proven is that these are the only relations satisfied by the generators; thus we want to check that we have no extra nilpotency or torsion conditions on $u$, and that the products $xu^i$ and $yu^i$ are nonzero for any $i$. For this, we look at the $2$-valuation of the characters of $Q_8$: we define an admissible approximation $(\widetilde{\Gamma}^n)$ that only takes into accounts some generators. This allows us to restrict ourselves when we compute the $2$-valuations of our evaluation morphisms, which we use to extract information about torsion in $R^*(G)$. 
	
	Let $X, Y, U$ be standard lifts of $x,y,u$. We consider the approximation $\{\widetilde{\Gamma}^n\}$, where $\widetilde{\Gamma}^n$ is the additive subgroup of $(R(G), +)$ generated by
	\begin{equation}
	X^{\epsilon_1}Y^{\epsilon_2}U^k, \text{ with } 2k+ \epsilon_1 + \epsilon_2 \geq n \text{ and } 0 \leq \epsilon_1+\epsilon_2 \leq 1. 
	\end{equation}
	
	\begin{Lemma}
		The approximation $(\widetilde{\Gamma}^n)$ is admissible.
	\end{Lemma}
	\begin{proof}
		Obviously we have $\widetilde{\Gamma}^{n+1} \subset \widetilde{\Gamma}^n$, so $(\widetilde{\Gamma}^n)$ satisfies condition  (\ref{admissiblefiltration1}). To check (\ref{admissiblefiltration2}), let $Z = \rho_3-1$ and $T = \Delta -2$ be lifts of $c_1(\rho_3)$, $c_1(\Delta)$, respectively. Let $\alpha$ be an additive generator for $\Gamma^n$. We know that $\Gamma^n$ is generated by products of Chern classes of the generating characters of $R(G)$, hence $\alpha$ is of the form:
		\[
		\alpha = X^{\epsilon_1}Y^{\epsilon_2}Z^{\epsilon_3}U^kT^l \in \Gamma^n, \ \ \ \ 2k+l+\epsilon_1+\epsilon_2+\epsilon_3 \geq n.
		\]
		If $\epsilon_{1} \geq 2$, then $\alpha$ contains a factor $X^2$, but the relation $x^2 = 0$ implies that $X^2 \in \Gamma^3$, so in that case $\alpha \in \Gamma^{n+1}$. The same goes for $\epsilon_2$, so we can restrict ourselves to monomials such that $\epsilon_1+\epsilon_2 \leq 1$. Similarly, since $c_1(\Delta) = 0$ we have $T \in \Gamma^2$, which means that $l \neq 0$ forces $\alpha \in \Gamma^{n+1}$. We proceed similarly for all factors and obtain
		\[
		\Gamma^n = \widetilde{\Gamma}^n + \Gamma^{n+1},
		\]
		so $(\widetilde{\Gamma}^n)$ satisfies condition (\ref{admissiblefiltration2}). 
	\end{proof}
	 There are 4 nontrivial evaluation morphisms on $R_\CC(Q_8)$:
	 \begin{itemize}
	 	\item $\phi_{-1}: \rho_i \mapsto 1, \Delta \mapsto -2$,
	 	\item $\phi_i: \rho_1 \mapsto 1, \ \ \ \rho_2,\rho_3 \mapsto -1, \ \ \ \Delta \mapsto 0$,
	 	\item $\phi_j: \rho_2 \mapsto 1, \ \ \ \rho_1,\rho_3 \mapsto -1, \ \ \ \Delta \mapsto 0$,
	 	\item $\phi_k: \rho_3 \mapsto 1, \ \ \ \rho_1, \rho_2 \mapsto -1, \ \ \ \Delta \mapsto 0$. 
	 \end{itemize}
	We apply those to our $X,Y,T,U$ and obtain:
	\begin{itemize}
		\item $\phi_{-1}: X,Y \mapsto 0, \ \ \ T \mapsto -4, \ \ \ U \mapsto 4$,
		\item $\phi_i: X \mapsto 0, \ \ \ Y,T \mapsto -2, \ \ \ U \mapsto 2$,
		\item $\phi_j: X, T \mapsto -2,  \ \ \ Y \mapsto 0, \ \ \ T \mapsto 2$,
		\item $\phi_k: X,Y,T \mapsto -2, \ \ \ T \mapsto 2$.
	\end{itemize}
	
	It is easy to check that, as stated in \Cref{proposition_pgroupscontinuous}, these morphisms are all continuous with respect to the $2$-adic topology. We can now wrap up the computation:
	\begin{proof}[Proof of \Cref{Q8_gradedring}.]
		We want to show that $R^{2n}(G) = \langle u^n \rangle = \ZZ/8\ZZ$, and $R^{2n+1}(G) = \langle xu^n, yu^n \rangle = (\ZZ/2\ZZ)^2$. \\
		We first look at $R^{2n}(G)$, where we need to show that $4u^n \neq 0$, that is, $4U^n \notin \Gamma^{2n+1}$. We have $\phi_g(4U^n) = 2^{n+2}$, for any $g = i,j,k$. Suppose that $4U^n \in \Gamma^{2n+1}$; then by \Cref{evaluation_approximation} there is an element $\widetilde{X} \in \widetilde{\Gamma}^{2n+1}$ satisfying
			\[
				n+2 = v_2(\phi_g(4U^n)) = v_2(\phi_g(\widetilde{X})). 
			\]
		Write 
			\begin{align*}
				\widetilde{X} =& a_1U^{n+1}+a_2XU^{n}+a_3YU^{n} \\
					&+ U^{n+2}P(X,Y,U) + XU^{n+1}Q(X,Y,U) + YU^{n+1}S(X,Y,U)
			\end{align*}
		where the $a_i$'s are integers and $P,Q,S$ are polynomials with integer coefficients. Apply $\phi_g$ for $g = i,j,k$ to this equation, divide each equation by $2^{n+1}$ and consider the result mod $2$. We obtain a system of three equations: 
		\begin{align*}
			0 &= a_1 + a_3 \\
			0 &= a_1 + a_2 \\
			0 &= a_1 + a_2 + a_3
		\end{align*}
		which has no nontrivial solution. But if $a_m = 0 (\mathrm{mod} \ 2)$ for $m=1,2,3$ then $v_2(\phi_g(\widetilde{x})) > n+2$, which is impossible. Thus $4U^n$ cannot be in $\Gamma^{2n+1}$, and the additive order of $u^m$ is indeed 8 for all $m$. \\
		The same process shows that $xu^n, yu^n$ and $xu^n+yu^n$ are nonzero elements of $R^{2n+1}(Q_8)$ for all $n$.
	\end{proof}

	\section{Abelian $2$-groups}\label{abelian2groups}
Abelian $2$-groups are a rich source of examples; we present below the computations of $R^*_\CC(C_4\times C_2)$ and $R^*_\CC(C_4\times C_4)$. To determine the graded rings in this section, we used all of the tools presented in this paper. \cref{theorem_quillenrelations} gives us a starting point to guess relations in the graded character rings, which we then determine precisely using basic virtual character algebra. Once we have a good candidate for $R^*(G)$, we discard as many extra relations as possible by restricting to various subgroups, and get rid of the last ones by looking at continuity of characters.
\begin{Remark}
	In the sequel, we denote the evaluation of a virtual character $X$ at $g \in G$ by $X\vert_{g}$ rather than $\phi_g(X)$. Besides, it will be convenient to use an additive notation throughout, so we denote the cyclic groups by $\ZZ/4$ and $\ZZ/2$.
\end{Remark}

	\begin{Proposition} \label{Z4Z2}
		\[
			R^*(\ZZ/4 \times \ZZ/2) = \frac{\ZZ[x,y]}{(4x,2y,xy^3+x^2y^2)}
		\]
		with $|x| = |y| = 1$.
	\end{Proposition}
	\begin{proof}
			Let $\rho$ be the generating character of $R(\ZZ/4)$ sending $1$ to $i$, and $\sigma$ the nontrivial representation of $\ZZ/2$. Then $R^*(G)$ is generated by $c_1(\rho) =: x$ and $c_1(\sigma) =: y$. By restriction to cyclic subgroups, we see that $x$ has additive order $4$ and $y$, additive order $2$. Now consider $X = \rho - 1$ and $Y = \sigma - 1$. Then by expanding $(X+1)^4 = 1$ we get
		\[
			4X = -(6X^2 + 4X^3 + X^4)
		\]
		On the other hand, $Y^n = (-2)^{n-1}Y$, so 
		\[
			XY^3 = 4XY = -X^4Y - 4X^3Y - 6X^2Y = -X^4Y - X^3Y^3 + 3X^2Y^2. 
		\]
		Modulo $\Gamma^5$, this is $xy^3 = x^2y^2$. The only other possible extra relations (that cannot be ruled out by restrictions to various subgroups) are: $x^{n-1}y = 0$, $x^{n-2}y^2 = 0$ or $x^{n-1}y = x^{n-2}y^2$ for some $n$. We use the continuity method to disprove all of these. Let $\widetilde{\Gamma}^n = \langle X^n, Y^n, X^{n-1}Y, X^{n-2}Y^2 \rangle$, then $\{ \widetilde{\Gamma}^n \}_n$ is an admissible approximation for $\{\Gamma^n\}_n$.\par
		First suppose that $x^{n-1}y = 0$ for some $n$, that is $X^{n-1}Y \in \Gamma^{n+1}$. Then for $N$ arbitrarily large, there exists $\widetilde{Z} \in \widetilde{\Gamma}^{n+1}$ such that $X^{n-1}Y = \widetilde{Z}+R$ with $R \in \Gamma^N$; in particular, by \Cref{evaluation_approximation}, for any $M>0$, there is an $N$ such that: 
		\[
			 \begin{cases} v_2\left(\widetilde{Z}\vert_{(k,\ell)}\right) = v_2\left(X^{n-1}Y\vert_{(k,\ell)}\right) \text{ whenever } v_2\left(X^{n-1}Y\vert_{(k,\ell)}\right) < \infty \\
			 	 v_2\left(\widetilde{Z}\vert_{(k,\ell)}\right) > M \text{ whenever } v_2\left(X^{n-2}Y^2\vert_{(k,\ell)}\right) = \infty
			 \end{cases},
		\]
		for all $(k,\ell) \in \ZZ/4\times\ZZ/2$.\par
		Write:
		\begin{align*}
			\widetilde{Z} =& a\cdot X^{n+1} + b\cdot X^{n}Y + c\cdot X^{n-1}Y^2 + d\cdot Y^{n+1} \\
				&+ P\cdot X^{n+2} + Q\cdot X^{n+1}Y + S\cdot X^{n}Y^2 + T\cdot Y^{n+2}
		\end{align*}
		where $a,b,c,d \in \ZZ$ and $P,Q,S,T \in \ZZ[X,Y]$. Evaluatin at $(2,1)$ gives $X\vert_{(2,1)} = Y\vert_{(2,1)} = -2$. Then the $2$-valuation of $X^{n-1}Y$ is $n$ while the $2$-valuation of $\widetilde{Z}$ is at least $n+1$. This shows that such a $\widetilde{Z}$ cannot exist, and thus $x^{n-1}y$ cannot be zero. A similar argument shows that $x^{n-2}y^2 \neq 0$.\par
		The only possible remaining relation is $x^{n-1}y = x^{n-2}y^2$. Let $Z = X^{n-1}Y + X^{n-2}Y^2$, and suppose $Z\in \Gamma^{n+1}$. Fix a large number $M>n+2$, and let $\widetilde{Z} \in \widetilde{\Gamma}^{n+1}$ and $R \in \Gamma^N$ satisfy $Z = \widetilde{Z} + R$ with the usual conditions on the valuation of $\widetilde{X}$. Then:
		\[
			\widetilde{Z}\vert_{(2,0)} = a\cdot(-2)^{n+1} + P\cdot (-2)^{n+2} 
		\]
		while $Z\vert_{(2,0)} = 0$, so $v_2(Z\vert_{(2,0)}) = +\infty$. In this case, we have $v_2(\widetilde{Z})>M>n+2$. This means in particular that $a = 0 \ (\Mod 2)$, hence $a = 2a'$ for some $a'$. \par
		We now evaluate at $(1,1)$:
		\[
			Z\vert_{(1,1)} = (i-1)^{n-2}\cdot 4 + (i-1)^{n-1}\cdot (-2).
		\]
		thus $v_2(Z\vert_{(1,1)}) = (n+1)/2$. On the other hand:
		\[
			\widetilde{Z}\vert_{(1,1)} = a'\cdot 2\cdot(i-1)^{n+1} + b\cdot(i-1)^{n}(-2) + c\cdot(i-1)^{n-1}4 + d\cdot(-2)^{n+1} + \widetilde{R}
		\]
		where $v_2(\widetilde{R}) \geq \frac{n+2}{2}$. We see that $v_2(\widetilde{Z}\vert_{(1,1)}) \geq (n+2)/2$, so we cannot have $v_2(\widetilde{Z}\vert_{(1,1)}) = \widetilde{Z}\vert_{(1,1)}$, in contradiction with our assumption. This means that $Z \notin \Gamma^{n+1}$, and thus  $x^{n-1}y \neq x^{n-2}y^2$. This completes the proof.
	\end{proof}
	
	\begin{Remark}
		Let $G = \ZZ/2^{n}\times \ZZ/2$. Then one can show, as above, that there is in $R^*(G)$ a relation of the form $ xy^{n+1} + x^2y^n = 0$. For $n=3$, it is possible to adapt the argument above and show that $R^*(G) = \frac{\ZZ[x,y]}{(8x, 2y, xy^{4} + x^2y^3)}$. Is it true in general that
		\[
			R^*(\ZZ/2^n \times \ZZ/2) = \frac{\ZZ[x,y]}{(2^nx, 2y, xy^{n+1} + x^2y^n)}?
		\]
	\end{Remark}
		
	\begin{Proposition} \label{Z4Z4}
		\[
			R^*(\ZZ/4 \times \ZZ/4) = \frac{\ZZ[x,y]}{(4x,4y,2x^2y+2xy^2, x^4y^2 - x^2y^4)}
		\]
		with $|x| = |y| = 1$
	\end{Proposition}
	\begin{proof}
		Existence of relations: let $X, Y$ be the usual lifts of $x,y$. We use that $(X+1)^4 = 1$, that is,
		\begin{equation}\label{Z4Z4relationX}
			X^4 = -4X - 6X^2 -4X^3
		\end{equation}
		Then: 
		\begin{align*}
			X^4Y &= (-4X - 6X^2 - 4X^3)Y = -4XY - 6X^2Y -4X^3Y \\
				&= X(6Y^2 + 4Y^3 + Y^4) - 6X^2Y -4X^3Y\\
			 	6XY^2 - 6X^2Y &= X^4Y - XY^4 +4X^3Y -4XY^3 \numberthis \label{Z4Z4relationdegree3}\\
				0 &= 2XY^2 + 2X^2Y (\Mod \Gamma^4) 
		\end{align*}
		so $2x^2y = 2xy^2$. By repeatedly applying \Cref{Z4Z4relationX} to $X^6Y - XY^6$, one obtains:
		\begin{equation*}
			X^6Y - XY^6 = 14X^2Y^3 - 14X^3Y^2 + 11X^2Y^4 - 11X^4Y^2 + 5X^3Y^4 - 5X^4Y^3. 
		\end{equation*}
		The expression $14X^2Y^3 - 14X^3Y^2$ is the sum of $8X^2Y^3 - 8X^3Y^2$ which belongs to $\Gamma^7$, and $6X^2Y^3 - 6X^3Y^2$ which is simply \Cref{Z4Z4relationdegree3} multiplied by $XY$, and thus also belongs to $\Gamma^7$. Thus $x^4y^2 = x^2y^4$.\\
		To show there are no extra relations, let $a_0, \cdots a_n \in \ZZ, n\geq 2$ satisfy: 
			\[
				z = a_0x^n + a_1x^{n-1}y + a_2x^{n-2}y^2 + a_3x^{n-3}y^3 + a_{n-1}xy^{n-1} + a_ny^n = 0
			\]
		or in other words,
			\begin{align*}
				Z &:= a_0X^n + a_1X^{n-1}Y + a_2X^{n-2}Y^2 + a_3X^{n-3}Y^3 + a_{n-1}XY^{n-1} + a_nY^n \\
					&= \sum_{k = 0}^{n+1}b_kX^{n+1-k}Y^k + \sum_{k = 0}^{n+2}P_k(X,Y)X^{n+2-k}Y^k \in \Gamma^{n+1}. \numberthis \label{Z4Z4_polynomial}
			\end{align*}
			Suppose, without loss of generality, that $a_0, a_1, a_n, b_0, b_1, b_n \in \{0,1,2,3\}$ and $a_2, a_3, a_{n-1}, b_2, \cdots, b_{n-1} \in \{0,1\}$ while $P_k(X,Y) \in \ZZ[X,Y]$. Let $\widetilde{Z}$ be the right hand side of the equation, and consider the restriction of \Cref{Z4Z4_polynomial} to the following subgroups:
			\begin{itemize}
				\item To $\ZZ/4\times 1$: then \Cref{Z4Z4_polynomial} becomes $a_0X^n = b_0X^{n+1} +P_0(X,0)X^{n+2} \in \Gamma^{n+1}$. This implies that $a_0 = b_0 = 0(\Mod 4)$, and since we had assumed that $a_0,b_0 \in \{0,1,2,3\}$ we have $a_0 = b_0 = 0$.
				\item To $1\times\ZZ/4$: similarly we obtain $a_n = b_n = 0$.
				\item To $\langle (1,1) \rangle \cong \ZZ/4$: the generators $X$ and $Y$ both restrict to the generator $T$ and we obtain $\sum a_kT^n = 0 (\Mod \Gamma^{n+1})$, so 
				\begin{equation}\label{Z4Z4congruence}
					a_1+a_2+a_3+a_{n-1} = 0 (\Mod 4).
				\end{equation}
			\end{itemize}
			Evaluating at $(1,2)$ yields:
			\begin{align*}
				v_2(\phi_{(1,2)}(Z)) &\geq \min\left(v_2(a_1) + \frac{n+1}{2}, v_2(a_2) + \frac{n+2}{2}, v_2(a_3) + \frac{n+3}{2}, v_2(a_{n-1}) + \frac{2n-1}{2}\right) \numberthis \label{Z4Z4valuation} \\
				v_2(\phi_{(1,2)}(\tilde{Z})) &\geq \frac{n+2}{2}
			\end{align*}
			with equality if there is a strict minimum in \Cref{Z4Z4valuation}, so we must have $v_2(a_1) \geq 1$, that is $a_1 = 0$ or $a_1 = 2$. By evaluating at $(2,1)$ instead, one shows that $v_2(a_{n-1}) \geq 1$ and thus $a_{n-1} = 0$, since we had assumed $a_{n-1} \in \{0,1\}$. Our equation becomes $a_1+a_2+a_3 = 0 (\Mod 4)$. If $a_1 = 0$ then $a_2 =a_3 = 0$, so we can assume $a_1 = 2$, and then $a_2 = a_3 = 1$. Thus $Z = 2X^{n-1}Y + X^{n-2}Y^2 + X^{n-3}Y^3$. To rule out this last possibility, consider the automorphism $\tau$ of $G$, which leaves $(1,0)$ invariant and sends $(0,1) $ to $(0,3)$. This induces a well-defined map $\tau^*$ of graded rings, which maps $y$ to $-y$. Then
			\[
				\tau^*z = -2x^{n-1}y + x^{n-2}y^2 - x^{n-3}y^3
			\]
			and summing $z + \tau^* z$ we obtain $2x^{n-2}y= 0$, which is impossible since all monomials in $R^*(G)$ have additive order $4$. Thus the assumption $a_0 = 2$ is wrong, and this concludes the proof.
	\end{proof}

\bibliographystyle{alpha}

\bibliography{bibliography}

\begin{thebibliography}{DdSMS99}

\bibitem[AT69]{atiyah-tall}
M.~F. Atiyah and D.~O. Tall.
\newblock Group representations, {$\lambda$}-rings and the {$J$}-homomorphism.
\newblock {\em Topology}, 8(3):253--297, July 1969.
\newblock MR:0244387. Zbl:0159.53301.

\bibitem[Ati61]{atiyah-characters_cohomology}
M.~F. Atiyah.
\newblock Characters and cohomology of finite groups.
\newblock {\em Inst. Hautes \'Etudes Sci. Publ. Math.}, (9):23--64, 1961.

\bibitem[Ati89]{atiyah-k-theory}
M.~F. Atiyah.
\newblock {\em {$K$}-theory}.
\newblock Advanced Book Classics. Addison-Wesley Publishing Company, Advanced
  Book Program, Redwood City, CA, second edition, 1989.
\newblock Notes by D. W. Anderson.

\bibitem[{Bea}01]{beauville}
A.~{Beauville}.
\newblock {Chern classes for representations of reductive groups}.
\newblock {\em ArXiv Mathematics e-prints}, April 2001.

\bibitem[Ber71]{berthelot-sga6}
P.~Berthelot.
\newblock Generalites sur les $\lambda$-anneaux.
\newblock In {\em Th{\'e}orie des Intersections et Th{\'e}or{\`e}me de
  Riemann-Roch}, pages 297--364, Berlin, Heidelberg, 1971. Springer Berlin
  Heidelberg.

\bibitem[DdSMS99]{dixon}
J.~D. Dixon, M.~P.~F. du~Sautoy, A.~Mann, and D.~Segal.
\newblock {\em Analytic pro-{$p$} groups}, volume~61 of {\em Cambridge Studies
  in Advanced Mathematics}.
\newblock Cambridge University Press, Cambridge, second edition, 1999.

\bibitem[FL85]{fulton-lang}
W.~Fulton and S.~Lang.
\newblock {\em Riemann-{R}och algebra}, volume 277 of {\em Grundlehren der
  Mathematischen Wissenschaften [Fundamental Principles of Mathematical
  Sciences]}.
\newblock Springer-Verlag, New York, 1985.

\bibitem[Ful98]{fulton-intersection}
W.~Fulton.
\newblock {\em Intersection theory}, volume~2 of {\em Ergebnisse der Mathematik
  und ihrer Grenzgebiete. 3. Folge. A Series of Modern Surveys in Mathematics
  [Results in Mathematics and Related Areas. 3rd Series. A Series of Modern
  Surveys in Mathematics]}.
\newblock Springer-Verlag, Berlin, second edition, 1998.

\bibitem[GM14]{guillot-minac}
P.~Guillot and J.~Min\'a\v{c}.
\newblock Milnor {$K$}-theory and the graded representation ring.
\newblock {\em J. K-Theory}, 13(3):447--480, 2014.

\bibitem[Qui68]{quillen}
D.~G. Quillen.
\newblock On the associated graded ring of a group ring.
\newblock {\em J. Algebra}, 10:411--418, 1968.

\bibitem[Ser77]{serre}
J.-P. Serre.
\newblock {\em Linear representations of finite groups}.
\newblock Springer-Verlag, 1977.
\newblock Translated from the second French edition by Leonard L. Scott,
  Graduate Texts in Mathematics, Vol. 42.

\end{thebibliography}

\end{document}